\documentclass
{amsart}
\usepackage{graphicx}
\newtheorem{thm}{Theorem}[section]
\theoremstyle{definition}
\newtheorem{cor}[thm]{Corollary}
\newtheorem{prop}[thm]{Proposition}
\newtheorem{defn}[thm]{Definition}

\newtheorem{rem}[thm]{Remark}
\newtheorem{ex}[thm]{Example}

\newtheorem{co}[thm]{Question}
\numberwithin{equation}{section}
 \usepackage[paperwidth=165mm, paperheight=235mm, twoside, hmargin={25mm,20mm}, vmargin={20mm,20mm} ]{geometry}
\begin{document}
\title[Nil-$M$-Noetherian and nil-$M$-Artinian modules]
{Nil-$M$-Noetherian and nil-$M$-Artinian modules}

\author[Faranak Farshadifar]%
{Faranak Farshadifar}

\newcommand{\acr}{\newline\indent}
\address{Department of Mathematics Education, Farhangian University, P.O. Box 14665-889, Tehran, Iran.}
\email{f.farshadifar@cfu.ac.ir}

\subjclass[2010]{13C13, 13C99}%
\keywords {Nil-$M$-cyclic, nil-$M$-Noetherian, nil-$M$-Artinian, nil-$M$-finitely generated, nil-$M$-finitely cogenerated}

\begin{abstract}
Let $R$ be a commutative ring with identity and $M$ be an $R$-module.
The aim of this paper is to introduce and investigate the notions of nil-$M$-Noetherian and nil-$M$-Artinian modules as generalizations of
Noetherian and Artinian modules. Also, in this regard we introduce nil versions of some algebraic concepts.
\end{abstract}
\maketitle
\section{Introduction}
\noindent
Throughout this paper, $R$ will denote a commutative ring with
identity and $\Bbb Z$ will denote the ring of integers. Also, $Nil(R)$ will denote the set of all nilpotent elements of $R$.

In \cite{SU18}, the authors introduced and investigated the notion of $\mathfrak{N}$-prime ideals as a generalization of prime ideals.
A proper ideal $P$ of $R$ is said to be a $\mathfrak{N}$-prime ideal if $ab \in P$, for each
$a, b \in R$, then either $a \in P + Nil(R)$ or $b \in P + Nil(R)$ \cite{SU18}.
In \cite{Faranak2024} (resp. \cite{Faranak2025}), the present author introduced the notion of nil-prime (resp. nil-primary) ideals as a generalization of prime (resp. primary) ideals and studied some basic properties of this class of ideals.
A proper ideal $P$ of $R$ is said to be a \textit{nil-prime ideal} if there exists $x \in Nil(R)$ such that whenever $ab \in P$,  then $a \in P$ or $b \in P$ or $a+x \in P$ or $b+x \in P$ for each $a, b \in R$ (for generalizations of this notion for modules, see \cite{Faranak2027}). Also, a proper ideal $P$ of $R$ is said to be a \textit{nil-primary ideal} if there exists $x \in Nil(R)$ such that whenever $ab \in P$ for some $a, b \in R$,  then  $a \in P$ or $b^n \in P$ or $a+x \in P$ or  $b^n+x \in P$ for some $n \in \Bbb N$.

For a submodule $N$ of an $R$-module $M$, a non-empty subset
$K$ of $M$, and a non-empty subset $J$ of $R$, the residuals of $N$ by $K$ and $J$ are
defined as $(N :_R K) =\{a \in R : aK \subseteq N\}$ and $(N :_M J) = \{m \in M : Jm\subseteq N\}$,
respectively. In particular, we use $Ann_R(M)$ to denote $(0 :_R M)$.

Motivated by nil-prime ideals, the purpose of this paper is to introduce and investigate the notions of nil-$M$-Noetherian and nil-$M$-Artinian modules as generalizations of Noetherian and Artinian modules. Also, in this regard we introduce the notions of nil-$M$-cyclic, nil-$M$-finitely generated and nil-$M$-finitely cogenerated modules.
\section{Main Results}
\begin{defn}\label{1.1}
We say that an $R$-module $M$ is a \textit{nil-$M$-Noetherian module} if there is $r \in \sqrt{Ann_R(M)}$ such that
for each ascending chain
$$
N_1\subseteq N_2\subseteq N_3\subseteq \ldots\subseteq N_k\subseteq \ldots
$$
of submodules of $M$ there is a positive integer $t$ such that $N_{t+i}\subseteq N_t+rM$ for all $i \geq 0$. Moreover, we say that
the ring $R$ is a \textit{nil-$R$-Noetherian ring} if $R$ as an $R$-module is a nil-$R$-Noetherian module.
\end{defn}

\begin{defn}\label{1.3}
We say that an $R$-module $M$ is a \textit{nil-$M$-Artinian module} if there is $r \in \sqrt{Ann_R(M)}$ such that
for each descending chain
$$
N_1\supseteq N_2\supseteq N_3\supseteq ...\supseteq N_k\supseteq \ldots
$$
of submodules of $M$ there is a positive integer $t$ such that $N_t\cap (0:_Mr)\subseteq N_{t+i}$ for all $i \geq 0$.
Also, we say that the ring $R$ is a \textit{nil-$R$-Artinian ring} if $R$ as an $R$-module is a nil-$R$-Artinian module.
\end{defn}

\begin{rem}\label{1.2}
As $0 \in \sqrt{Ann_R(M)}$, we have that every Noetherian (resp. Artinian) $R$-module is a nil-$M$-Noetherian (resp. nil-$M$-Artinian) $R$-module. This motivates the Question \ref{1.4}.
If $\sqrt{Ann_R(M)}=0$, then the notions of Noetherian $R$-modules and nil-$M$-Noetherian $R$-modules (resp. Artinian $R$-modules and nil-$M$-Artinian $R$-modules) are equal. For example, for a positive integer $n$ which is square-free (i.e., $n$ has not a square factor),   consider the ring $\Bbb Z_n$ of integers modulo $n$. Then $\sqrt{Ann_{\Bbb Z_n}(\Bbb Z_n)}=0=Nil(\Bbb Z_n)$. Also, $\sqrt{Ann_{\Bbb Z}( \Bbb Z)}=0=Nil(\Bbb Z)$ and $\sqrt{Ann_{F[x])}(F[x]))}=0=Nil(F[x])$,  where $F$ is a field.
\end{rem}

\begin{co}\label{1.4}
Let $M$ be an $R$-module.
Is every nil-$M$-Noetherian (resp. nil-$M$-Artinian) module a Noetherian (resp. Artinian) module?
\end{co}

\begin{defn}\label{222.1}
We say that a submodule $N$ of an $R$-module $M$ is a \textit{nil-$M$-cyclic} if there exist $r \in \sqrt{Ann_R(M)}$ and $x \in M$ such that $N=Rx+rM$.
\end{defn}

Let $M$ be an $R$-module. Since $0 \in \sqrt{Ann_R(M)}$, clearly each cyclic submodule of $M$ is a nil-$M$-cyclic submodule of $M$. But  Example \ref{e2.1} shows that the converse is not true in general.
\begin{ex}\label{e2.1}
The submodule $N=\Bbb Zx+2\Bbb Z_4[x,y]$ of the $\Bbb Z$-module $M=\Bbb Z_4[x,y]$ is a nil-$M$-cyclic since $2\in \sqrt{Ann_\Bbb Z(M)}$ and $x \in \Bbb Z_4[x,y]$. But $N$ is not a cyclic submodule of $M$.
\end{ex}

\begin{defn}\label{22.1}
We say that a submodule $N$ of an $R$-module $M$ is a \textit{nil-$M$-finitely generated} (resp. \textit{weakly nil-$M$-finitely generated}) if there exist $r \in \sqrt{Ann_R(M)}$ and $x_1,x_2, \ldots, x_t \in M$ such that $N=Rx_1+Rx_2+\cdots+Rx_t+rM$ (resp. $N+rM=Rx_1+Rx_2+\cdots+Rx_t+rM$).
\end{defn}

\begin{rem}\label{282.1}
Let $M$ be an $R$-module. Clearly, every nil-$M$-finitely generated submodule of $M$ is a weakly nil-$M$-finitely generated submodule of $M$. If $N$ is a weakly nil-$M$-finitely generated submodule of $M$ such that $(N:_RM)$ is prime ideal of $R$, then
$\sqrt{Ann_R(M)}\subseteq \sqrt{(N:_RM)}=(N:_RM)$. This in turn implies that
 $N$ is a nil-$M$-finitely generated submodule of $M$.
\end{rem}

\begin{prop}\label{69.3}
Let $S$ be a multiplicatively subset of $R$ and let $M$ be an $R$-module.
If $N$ is a nil-$M$-finitely generated submodule of $M$, then  $S^{-1}N$ is a nil-$S^{-1}M$-finitely generated submodule of $S^{-1}M$.
\end{prop}
\begin{proof}
Assume that $N$ is a nil-$M$-finitely generated submodule of $M$. Then there exist $r \in \sqrt{Ann_R(M)}$ and $x_1,x_2,\ldots , x_t \in M$ such that $N=Rx_1+Rx_2+\cdots+Rx_t+rM$. This implies that $r/1\in  S^{-1}(\sqrt{Ann_R(M)})\subseteq  \sqrt{Ann_{S^{-1}R}(S^{-1}M)}$ and $x_1/1,x_2/1,\ldots, x_t/1 \in S^{-1}M$ such that
$$
S^{-1}N=(S^{-1}R)(x_1/1)+(S^{-1}R)(x_2/1)+\cdots+(S^{-1}R)(x_t/1)+r/1(S^{-1}M).
$$
\end{proof}

A proper submodule $N$ of $M$ is said to be \emph{completely irreducible} if $N=\bigcap _
{i \in I}N_i$, where $ \{ N_i \}_{i \in I}$ is a family of
submodules of $M$, implies that $N=N_i$ for some $i \in I$. By \cite{FHo06},
every submodule of $M$ is an intersection of completely irreducible submodules of $M$. Thus the intersection
of all completely irreducible submodules of $M$ is zero \cite{FHo06}.

\begin{defn}\label{2422.1}
We say that a submodule $N$ of an $R$-module $M$ is a \textit{nil-$M$-completely irreducible submodule} if there exist $r \in \sqrt{Ann_R(M)}$ and completely irreducible submodule $L$ of $M$ such that $N=L\cap (0:_Mr)$.
\end{defn}

\begin{defn}\label{27882.1}
For a submodule $N$ of an $R$-module $M$, we say that $M/N$ is a \textit{nil-$M$-finitely cogenerated} (resp. \textit{weakly nil-$M$-finitely cogenerated}) if there exists $r \in \sqrt{Ann_R(M)}$ such that
$N=\cap_{i \in I}N_i$ implies that $N=\cap_ {i\in J}N_i\cap (0:_Mr)$ (resp. $\cap_ {i\in J}N_i\cap (0:_Mr)= N \cap (0:_Mr)$) for some finite subset $J$ of $I$.
\end{defn}

\begin{prop}\label{272.1}
If for a submodule $N$ of an $R$-module $M$ we have $M/N$ is a nil-$M$-finitely cogenerated (resp. weakly nil-$M$-finitely cogenerated) module, then there exist $r \in \sqrt{Ann_R(M)}$ and completely irreducible submodules $L_1,L_2, \ldots,L_t$ of $M$ such that $N=L_1 \cap L_2\cap \cdots \cap L_t\cap (0:_Mr)$ (resp. $L_1 \cap L_2\cap \cdots \cap L_t\cap (0:_Mr)= N\cap (0:_Mr)$).
\end{prop}
\begin{proof}
This follows from the fact that every submodule of $M$ is an intersection of completely irreducible submodules of $M$.
\end{proof}

\begin{rem}\label{2892.1}
Let $M$ be an $R$-module. Clearly, every nil-$M$-finitely cogenerated homomorphic image of $M$ is a weakly nil-$M$-finitely cogenerated homomorphic image of $M$. It is easy to see that, if $M/N$ is a weakly nil-$M$-finitely cogenerated homomorphic image of $M$ such that $Ann_R(N)$ is prime ideal of $R$, then $M/N$ is a nil-$M$-finitely cogenerated homomorphic image of $M$.
\end{rem}

\begin{prop}\label{1.71}
Let $M$ be an $R$-module. Then we have the following.
\begin{itemize}
\item [(a)] If for a submodule $N$ of $M$ and $r \in \sqrt{Ann_R(M)}$ we have $N+rM=M$, then $N=M$.
\item [(b)] If for a submodule $N$ of $M$ and $r \in \sqrt{Ann_R(M)}$ we have $N\cap (0:_Mr)=0$, then $N=0$.
\end{itemize}
\end{prop}
\begin{proof}
(a) Suppose that for a submodule $N$ of $M$ and $r \in \sqrt{Ann_R(M)}$ we have $N+rM=M$. Then there exists $t \in \Bbb N$ such that $r^tM=0$. We have $rN+r^2M=rM$. This implies that  $rN+r^2M+N=rM+N$ and so $N+r^2M=M$. Continuing this way, we get that
$M=N+r^tM=N+0=N$.

(b) Suppose that for a submodule $N$ of $M$ and $r \in \sqrt{Ann_R(M)}$ we have $N\cap (0:_Mr)=0$. Then there exists $t \in \Bbb N$ such that $r^tM=0$. Clearly, $N\cap (0:_Mr)=(0:_Nr)$. Hence, $(0:_Nr)=0$. Therefore, 
$$
(0:_Nr^2)=((0:_Nr):_Nr)=(0:_Nr)=N \cap (0:_Mr)=0.
$$
Continuing this way, we have
$0=(0:_Nr^t)=N\cap (0:_Mr^t)=N\cap M=
N$.
\end{proof}

\begin{cor}\label{2t22.1}
Let $M$ be an $R$-module. Then we have the following.
\begin{itemize}
\item [(a)] $M$ is a nil-$M$-cyclic module if and only if it is a cyclic module.
\item [(b)] $M$ is a weakly nil-$M$-finitely generated module if and only if it is a finitely generated module.
\item [(c)] The submodule $0$ of $M$ is a  nil-$M$-completely irreducible submodule
if and only if it is a completely irreducible submodule.
\item [(d)] $M$ is a weakly nil-$M$-finitely cogenerated module
if and only if it is a finitely cogenerated module.
\end{itemize}
\end{cor}
\begin{proof}
(a)  Let $M$ be a nil-$M$-cyclic $R$-module. Then there exist $r \in \sqrt{Ann_R(M)}$ and $x \in M$ such that $M=Rx+rM$.
This implies that $M=Rx$ by Proposition \ref{1.71} (a). Thus $M$ is a cyclic module. Since $0 \in \sqrt{Ann_R(M)}$, the converse is clear.

(b) Let $M$ be a weakly nil-$M$-finitely generated $R$-module. Then there exist $r \in \sqrt{Ann_R(M)}$ and $x_1,x_2,\ldots, x_t \in M$ such that $M=M+rM=Rx_1+Rx_2+\cdots+Rx_t+rM$. Thus $M=Rx_1+Rx_2+\cdots+Rx_t$ by Proposition \ref{1.71} (a). So, $M$ is a finitely generated module. Since $0 \in \sqrt{Ann_R(M)}$, the converse is clear.

(c) Assume that the submodule $0$ of $M$ is a nil-$M$-completely irreducible submodule. Then there exist $r \in \sqrt{Ann_R(M)}$ and completely irreducible submodule $L$ of $M$ such that $0=L\cap (0:_Mr)$.
This implies that $L=0$ by Proposition \ref{1.71} (b). So, $0$ is a completely irreducible submodule of $M$. Since $0 \in \sqrt{Ann_R(M)}$, the converse is clear.

(d) Suppose that $M$ is a weakly nil-$M$-finitely cogenerated $R$-module and $0=\cap_{i \in I} N_i$. Then there exist $r \in \sqrt{Ann_R(M)}$ and finite subset $J$ of $I$ such that $\cap_{i\in J}N_i\cap (0:_Mr)=0\cap (0:_Mr)=0$. It follows that  $\cap_{i\in J}N_i=0$ by Proposition \ref{1.71} (a). Therefore, $M$ is a finitely cogenerated module. Since $0 \in \sqrt{Ann_R(M)}$, the converse is clear.
\end{proof}

\begin{thm}\label{6.1}
Let $M$ be an $R$-module. Consider the following statements:
\begin{itemize}
\item [(a)] Every submodule of $M$ is nil-$M$-finitely generated;
\item [(b)] $M$ is a nil-$M$-Noetherian module;
\item [(c)] There exists $r \in \sqrt{Ann_R(M)}$ such that any non-empty collection $\Sigma$ of submodules of $M$ with the form $K+rM$ has a maximal element $N+rM$;
\item [(d)] Every submodule of $M$ is weakly nil-$M$-finitely generated.
\end{itemize}
Then $(a)\Rightarrow (b)$, $(b)\Rightarrow (c)$, and $(c)\Rightarrow (d)$.
\end{thm}
\begin{proof}
$(a)\Rightarrow (b)$
 Assume that the following is an increasing sequence of submodules of $M$
$$
N_1\subseteq N_2\subseteq N_3\subseteq \ldots\subseteq N_k\subseteq \ldots.
$$
Consider $N := \bigcup_{i\geq1} N_i$. Then $N$ is a submodule of $M$, hence it is nil-$M$-finitely generated by part (a). Thus there exist
 $r \in \sqrt{Ann_R(M)}$ and $u_1, \ldots , u_t \in M$ such that $N=Ru_1+\cdots +Ru_t+rM$. Then we can find $s$ such that $u_j \in N_s$ for all $j=1,2,\ldots, t$. In this case we have $N+rM \subseteq N_s+rM$. Therefore,  for all $i \geq 0$ we have
$
N_{s+i}\subseteq N\subseteq N+rM \subseteq N_s+rM
$,
as needed.

$(b)\Rightarrow (c)$
First note that by part (b), there is  $r \in \sqrt{Ann_R(M)}$ such that
for each ascending chain
$$
N_1\subseteq N_2\subseteq N_3\subseteq \ldots\subseteq N_k\subseteq \ldots
$$
of submodules of $M$ there is a positive integer $t$ such that $N_{t+i}\subseteq N_t+rM$ for all $i \geq 0$.
 Suppose that there exists a non-empty collection $\Sigma$ of submodules of $M$ with the form $N+rM$ such that it has no maximal element. Let us choose $N_1 +rM\in \Sigma$. Since this is not maximal, there is $N_2 +rM\in \Sigma$ such that $N_1+rM \subset N_2+rM$, and we continue in this way to construct the following infinite strictly increasing sequence of submodules of $M$:
 $$
 N_1+rM \subset N_2+rM\subset N_3+rM \subset \ldots.
  $$
 By part (b), there is a positive integer $t$ such that
 $$
 N_{t+1}+rM\subseteq N_t+rM+rM=N_t+rM,
 $$
which is a contradiction.

$(c)\Rightarrow (d)$
 Suppose that there exists $r \in \sqrt{Ann_R(M)}$ such that any non-empty collection $\Sigma$ of submodules of $M$ with the form $K+rM$ has a maximal element.
Let $N$ be a submodule of $M$ and consider the family $\Sigma$ of all submodules of $M$ with the form $K+rM$, where $K=Rx_1+\cdots +Rx_t$ for some $x_i\in N$. Clearly, $\Sigma$ is non-empty. By part (c), $\Sigma$ has a maximal element $\acute{N}+rM$,  where  $\acute{N}=Rx_1+\cdots +Rx_n$ for some $x_i \in N$. If $\acute{N}+rM\not=N+rM$, then there is $u \in  N+rM\setminus \acute{N}+rM$ and the submodule $Ru+\acute{N}+rM$ strictly containing $\acute{N}+rM$, a contradiction. Thus $\acute{N}+rM= N+rM$, as needed.
 \end{proof}

\begin{defn}\label{1.1999}
\begin{itemize}
\item [(a)] We say that a submodule $N$ of an $R$-module $M$ is a \textit{nil-$M$-Noetherian submodule} if there is $r \in \sqrt{Ann_R(M)}$ such that
for each ascending chain
$$
N_1\subseteq N_2\subseteq N_3\subseteq \ldots\subseteq N_k\subseteq \ldots
$$
of submodules of $N$ there is a positive integer $t$ such that $N_{t+i}\subseteq N_t+rM$ for all $i \geq 0$.

\item [(b)] We say that a quotient module $M/N$ of an $R$-module $M$ is a \textit{nil-$M$-Noetherian quotient module} if there is $r \in \sqrt{Ann_R(M)}$ such that
for each ascending chain
$$
N_1/N\subseteq N_2/N\subseteq N_3/N\subseteq \ldots\subseteq N_k/N\subseteq \ldots
$$
of submodules of $M/N$ there is a positive integer $t$ such that $N_{t+i}/N\subseteq (N_t+rM)/N$ for all $i \geq 0$.
\end{itemize}
\end{defn}

Recall that a submodule $N$ of an $R$-module $M$ is said to be \textit{pure} in $M$ if $rM\cap N = rN$ for each $r\in R$ \cite{Ri72}.
$M$ is said to be \emph{fully pure} if every submodule of $M$ is pure \cite{AF122}.
\begin{prop}\label{6.2}
Let $M$ be a nil-$M$-Noetherian $R$-module. Then we have the following.
\begin{itemize}
\item [(a)] Any pure submodule $N$ of $M$ is a nil-$N$-Noetherian $R$-module.
\item [(b)] Any quotient module $M/N$ of $M$ is  a nil-$M/N$-Noetherian $R$-module.
\item [(c)] Any submodule $N$ of $M$ is a nil-$M$-Noetherian $R$-module.
\item [(d)] Any quotient module $M/N$ of $M$ is  a nil-$M$-Noetherian $R$-module.
\end{itemize}
\end{prop}
\begin{proof}
(a) Let $N$ be a pure submodule of $M$ and
$$
K_1\subseteq K_2\subseteq K_3\subseteq \ldots\subseteq K_t\subseteq \ldots
$$
be an ascending chain of submodules of $N$. Then since this chain is also an ascending chain of submodules of $M$,
there is a positive integer $s$ such that $K_{s+i}\subseteq K_s+rM$ for all $i \geq 0$, where $r \in \sqrt{Ann_R(M)}\subseteq \sqrt{Ann_R(N)}$. As $N$ is pure, $N\cap rM=rN$. Therefore, for $r \in \sqrt{Ann_R(N)}$ we get that
$$
K_{s+i}=K_{s+i}\cap N\subseteq (K_s+rM)\cap N=(K_s\cap N)+(rM\cap N)=K_s+rN,
$$
for all $i \geq 0$. Thus $N$ is a nil-$N$-Noetherian $R$-module.

(b) Let $N$ be a submodule of $M$ and
$$
H_1/N\subseteq H_2/N\subseteq H_3/N\subseteq \ldots\subseteq H_t/N\subseteq \ldots
$$
be an ascending chain of submodules of $M/N$. Then $$
H_1\subseteq H_2\subseteq H_3\subseteq \ldots\subseteq H_t\subseteq \ldots
$$
is an ascending chain of submodules of $M$. Thus there is a positive integer $t$ such that $H_{t+i}\subseteq H_t+rM$ for all $i \geq 0$, where $r \in \sqrt{Ann_R(M)}\subseteq \sqrt{Ann_R(M/N)}$.  This implies that for all $i \geq 0$,
$H_{t+i}/N\subseteq H_t/N+r(M/N)$, where $r \in  \sqrt{Ann_R(M/N)}$. Hence, $M/N$ is a nil-$M/N$-Noetherian $R$-module.

(c) This follows from the fact that any increasing sequence of submodules of $N$ is also an increasing sequence of submodules of $M$.

(d) This follows from the fact that any increasing sequence of submodules of $M/N$ corresponds to a sequence of submodules of $M$ containing $N$, So $M/N$ is a nil-$M$-Noetherian module.
\end{proof}

\begin{thm}\label{69.1}
Let $M$ be an $R$-module. Consider the following statements:
\begin{itemize}
\item [(a)] Every homomorphic image of $M$ is nil-$M$-finitely cogenerated $R$-module;
\item [(b)] $M$ is a nil-$M$-Artinian module;
\item [(c)] There exists $r \in \sqrt{Ann_R(M)}$ such that any non-empty collection $\Sigma$ of submodules of $M$ with the form $K\cap (0:_Mr)$ has a minimal element $N\cap (0:_Mr)$.
\item [(d)] Every homomorphic image of $M$ is weakly nil-$M$-finitely cogenerated.
\end{itemize}
Then $(a)\Rightarrow (b)$, $(b)\Rightarrow (c)$, and $(c)\Rightarrow (d)$.
\end{thm}
\begin{proof}
$(a)\Rightarrow (b)$
 Suppose that the following is a descending chain of submodules of $M$
$$
N_1\supseteq N_2\supseteq N_3\supseteq \ldots\supseteq N_k\supseteq \ldots.
$$
Consider $N := \bigcap_{i\geq1} N_i$. Then $N$ is a submodule of $M$, hence $M/N$ is nil-$M$-finitely generated by part (a). Thus there exist $r \in \sqrt{Ann_R(M)}$ and $N_{j_1}, \ldots, N_{j_t}$ of the above chain such that $N=N_{j_1}\cap \ldots \cap N_{j_t}\cap (0:_Mr)$.
Then we can find $s$ such that $N_{j_s} \subseteq N_{j_i}$ for all $i=1,\ldots , t$.
Therefore, $N=N_{j_s}\cap (0:_Mr)$. Now, for all $i \geq 0$ we have
$
N_{j_s}\cap (0:_Mr) = N \subseteq N_{j_{s+i}}
$,
as needed.

$(b)\Rightarrow (c)$
First note that by part (b), there is  $r \in \sqrt{Ann_R(M)}$ such that
for each descending chain
$$
N_1\supseteq N_2\supseteq N_3\supseteq ...\supseteq N_k\supseteq \ldots
$$
of submodules of $M$ there is a positive integer $t$ such that $N_t\cap (0:_Mr)\subseteq N_{t+i}$ for all $i \geq 0$.
Suppose that there exists a non-empty collection $\Sigma$ of submodules of $M$ with the form $N\cap (0:_Mr)$ such that it has no minimal element. Let us choose $N_1 \cap (0:_Mr)\in \Sigma$. Since this is not minimal, there is $N_2 \cap (0:_Mr)\in \Sigma$ such that $N_2\cap (0:_Mr) \subset N_1\cap (0:_Mr)$, and we continue in this way to construct the following infinite strictly descending chains of submodules of $M$:
 $$
 N_1\cap (0:_Mr) \supset N_2\cap (0:_Mr)\supset N_3\cap (0:_Mr) \supset \ldots.
  $$
 By part (b), there is a positive integer $t$ such that
 $$
 N_t\cap (0:_Mr)=N_t\cap (0:_Mr)\cap (0:_Mr)\subseteq N_{t+1}\cap (0:_Mr),
 $$
which is a contradiction.

$(c)\Rightarrow (d)$
 Suppose that there exists $r \in \sqrt{Ann_R(M)}$ such that any non-empty collection $\Sigma$ of submodules of $M$ with the form $K\cap (0:_Mr)$ has a minimal element.
Let $N$ be a submodule of $M$ and $N=\cap_{i\in I} N_i$.
Consider the family $\Sigma$ of all submodules of $M$ with the form $\cap_{i\in J} N_i\cap (0:_Mr)$, where
$J$ is a finite subset of $I$.
Clearly, $\Sigma$ is non-empty. By part (c), $\Sigma$ has a minimal element $\cap_{i\in H} N_i\cap (0:_Mr)$,  where  $H$ is a finite subset of $I$.  By the minimality, we get that $N\cap (0:_Mr)=\cap_{i\in I} N_i\cap (0:_Mr)=\cap_{i\in H} N_i\cap (0:_Mr)$, as needed.
 \end{proof}

 \begin{cor}\label{677.92}
Let $M$ be an $R$-module. Then we have the following.
\begin{itemize}
\item [(a)] If $M$ is a nil-$M$-Noetherian $R$-module, then $M$ is a finitely generated $R$-module.
\item [(b)] If $M$ is a nil-$M$-Artinian $R$-module, then $M$ is a finitely cogenerated $R$-module.
\end{itemize}
\end{cor}
\begin{proof}
(a) Assume that $M$ is a nil-$M$-Noetherian $R$-module. Then $M$ is a weakly nil-$M$-finitely generated module by Theorem \ref{6.1}.
Now the result follows from Corollary \ref{2t22.1} (b).

(b) Suppose that $M$ is a nil-$M$-Artinian $R$-module. Then $M$ is a weakly nil-$M$-finitely cogenerated module by Theorem \ref{69.1}.
Now the result follows from Corollary \ref{2t22.1} (d)
\end{proof}

\begin{defn}\label{1.1999}
\begin{itemize}
\item [(a)] We say that a submodule $N$ of an $R$-module $M$ is a \textit{nil-$M$-Artinian submodule} if there is $r \in \sqrt{Ann_R(M)}$ such that for each descending chain
$$
N_1\supseteq N_2\supseteq N_3\supseteq ...\supseteq N_k\supseteq \ldots
$$
of submodules of $N$ there is a positive integer $t$ such that $N_t\cap (0:_Mr)\subseteq N_{t+i}$ for all $i \geq 0$.

\item [(b)] We say that a quotient module $M/N$ of an $R$-module $M$ is a \textit{nil-$M$-Artinian quotient module}  if there is $r \in \sqrt{Ann_R(M)}$ such that
for each descending chain
$$
N_1/N\supseteq N_2/N\supseteq N_3/N\supseteq ...\supseteq N_k/N\supseteq \ldots
$$
of submodules of $M/N$ there is a positive integer $t$ such that $(N_t\cap (0:_Mr))/N\subseteq N_{t+i}/N$ for all $i \geq 0$.
\end{itemize}
\end{defn}
A submodule $N$ of an $R$-module $M$ is said to be \textit{copure} in $M$ if $(N:_MI)=N+(0:_MI)$ for each ideal $I$ of $R$ \cite{AF09}.
$M$ is said to be \emph{fully copure} if every submodule of $M$ is copure \cite{AF122}.
\begin{prop}\label{677.9299}
Let $M$ be a nil-$M$-Artinian $R$-module. Then we have the following.
\begin{itemize}
\item [(a)] Any submodule $N$ of $M$ is a nil-$N$-Artinian $R$-module.
\item [(b)] Any quotient module $M/N$ of $M$ is  a nil-$M/N$-Artinian $R$-module, where $N$ is a copure submodule of $M$.
\item [(c)] Any submodule $N$ of $M$ is a nil-$M$-Artinian $R$-module.
\item [(d)] Any quotient module $M/N$ of $M$ is  a nil-$M$-Artinian $R$-module.
\end{itemize}
\end{prop}
\begin{proof}
(a) Let $N$ be a submodule of $M$ and
$$
K_1\supseteq K_2\supseteq K_3\supseteq \ldots\supseteq K_t\supseteq \ldots
$$
be a descending chain of submodules of $N$. Then since this chain is also a descending chain of submodules of $M$,
there is a positive integer $s$ such that $K_s \cap (0:_Mr)\subseteq K_{s+i}$ for all $i \geq 0$, where $r \in \sqrt{Ann_R(M)}\subseteq \sqrt{Ann_R(N)}$. This implies that  $K_s \cap (0:_Nr)=N\cap K_s \cap (0:_Mr)\subseteq N\cap K_{s+i}=K_{s+i}$ for $r \in \sqrt{Ann_R(N)}$.

(b) Let $N$ be a copure submodule of $M$ and
$$
H_1/N\supseteq H_2/N\supseteq H_3/N\supseteq \ldots\supseteq H_t/N\supseteq \ldots
$$
be a descending chain of submodules of $M/N$. Then we have the following descending chain of submodules of $M$.
$$
H_1\supseteq H_2\supseteq H_3\supseteq \ldots\supseteq H_t\supseteq \ldots.
$$
Thus there is a positive integer $t$ such that $H_t\cap (0:_Mr)\subseteq H_{t+i}$ for all $i \geq 0$, where $r \in \sqrt{Ann_R(M)}\subseteq \sqrt{Ann_R(M/N)}$. Since $N$ is copure, $(N:_Mr)= N+(0:_Mr)$. Thus we have
$$
H_t\cap (N:_Mr)=H_t\cap (N+(0:_Mr))=(H_t\cap (0:_Mr))+(N\cap H_t)
$$
$$
= (H_t\cap (0:_Mr))+N\subseteq H_{t+i}+N=H_{t+i}.
$$
This implies that for all $i \geq 0$,
$H_t/N\cap (0:_{M/N}r)\subseteq H_{t+i}/N$, where $r \in  \sqrt{Ann_R(M/N)}$. Hence, $M/N$ is a nil-$M/N$-Artinian $R$-module.

(c) This follows from the fact that any descending chain of submodules of $N$ is also a descending chain of submodules of $M$.

(d) This follows from the fact that any descending chain of submodules of $M/N$ corresponds to a descending chain of submodules of $M$ containing $N$, So $M/N$ is a nil-$M$-Artinian module.
\end{proof}

\begin{thm}\label{677.9288}
Let $M$ be an $R$-module. Then we have the following.
\begin{itemize}
\item [(a)] If $M$ is a fully pure nil-$M$-Noetherian $R$-module, then $M$ is a Noetherian $R$-module.
\item [(b)] If $M$ is a fully copure nil-$M$-Artinian $R$-module, then $M$ is an Artinian $R$-module.
\end{itemize}
\end{thm}
\begin{proof}
(a) Assume that $M$ is a fully pure nil-$M$-Noetherian $R$-module. Then by Proposition \ref{6.2} (a), every submodule $N$ of $M$ is a
nil-$N$-Noetherian $R$-module. Now, by Corollary \ref{677.92} (a), every submodule $N$ of $M$ is a finitely generated. Thus $M$ is a Noetherian $R$-module.

(b) Assume that $M$ is a fully copure nil-$M$-Artinian $R$-module. Then by Proposition \ref{677.9299} (b), for every submodule $N$ of $M$ we have $M/N$ is a
nil-$M/N$-Artinian $R$-module. Now, by Corollary \ref{677.92} (b), $M/N$ is a finitely generated. Hence, $M$ is an Artinian $R$-module.
\end{proof}

Recall that an $R$-module $M$ is said to be a
\emph{multiplication module} if for every submodule $N$ of $M$
there exists an ideal $I$ of $R$ such that $N=IM$ \cite{Ba81}. An $R$-module $M$ is said to be a \emph{comultiplication module} if for every submodule $N$ of $M$ there exists an ideal $I$ of $R$ such that $N=(0:_MI)$ \cite{AF07}.
It is easy to see that $M$ is a multiplication (resp. comultiplication) module if and only if $N=(N:_RM)M$ (resp. $N=(0:_MAnn_R(N))$) for each submodule $N$ of $M$.

\begin{prop}\label{699.4}
Let $M$ be an $R$-module. Then we have the following.
\begin{itemize}
\item [(a)] If $R$ is a nil-$R$-Noetherian ring and $M$ is a multiplication $R$-module, then $M$ is a nil-$M$-Noetherian module.
\item [(b)] If $R$ is a nil-$R$-Artinian ring and $M$ is a faithful comultiplication $R$-module, then $M$ is a nil-$M$-Noetherian module.
\end{itemize}
\end{prop}
\begin{proof}
(a) Let
$$
K_1\subseteq K_2\subseteq K_3\subseteq \ldots\subseteq K_t\subseteq \ldots
$$
be an ascending chain of submodules of $M$.
Then we have
$$
(K_1:_RM)\subseteq (K_2:_RM)\subseteq (K_3:_RM)\subseteq \ldots\subseteq (K_t:_RM)\subseteq \ldots
$$
is an ascending chain of ideals of $R$. Thus by assumption, there exist $r \in \sqrt{0}\subseteq \sqrt{Ann_R(M)}$ and a positive integer $t$ such that $(N_{t+i}:_RM)\subseteq (N_t:_RM)+rR$ for all $i \geq 0$. This implies that
$$
N_{t+i}=(N_{t+i}:_RM)M\subseteq (N_t:_RM)M+rM=N_t+rM
$$
for all $i \geq 0$, as needed.

(b) Let
$$
K_1\subseteq K_2\subseteq K_3\subseteq \ldots\subseteq K_t\subseteq \ldots
$$
be an ascending chain of submodules of $M$.
Then we have
$$
Ann_R(K_1)\supseteq Ann_R(K_2)\supseteq Ann_R(K_3)\supseteq \ldots\supseteq Ann_R(K_t)\supseteq \ldots
$$
 is a descending chain of ideals of $R$. Thus by assumption, there exist $r \in \sqrt{0}\subseteq \sqrt{Ann_R(M)}$ and a positive integer $t$ such that $Ann_R(N_t)\cap (0:_Rr)\subseteq Ann_R(N_{t+i})$ for all $i \geq 0$.
As $M$ is faithful, $Ann_R(r)=Ann_R(rM)$. Thus by using \cite[Proposition 3.3]{AF08}, we have
$$
N_{t+i}=(0:_MAnn_R(N_{t+i})) \subseteq (0:_MAnn_R(N_t)\cap (0:_Rr))=
$$
$$
(0:_MAnn_R(N_t)\cap Ann_R(rM))=N_t+rM.
$$
for all $i \geq 0$, as needed.
\end{proof}

\begin{defn}\label{699.18}
Let $M$ be an $R$-module and $f$ be an endomorphism of $M$. We say that $f$ is \textit{nil-epic} if there exists  $r \in \sqrt{Ann_R(M)}$ such that $(0:_Mr) \subseteq Im f$. Also, we say that $f$ is \textit{nil-monic} if there exists $r \in \sqrt{Ann_R(M)}$ such that $Ker f \subseteq rM$.
\end{defn}

\begin{thm}\label{699.14}
Let $M$ be an $R$-module and $f$ be an endomorphism of $M$.
\begin{itemize}
\item [(a)] If $M$ is a nil-Artinian module, then there exists $r \in \sqrt{Ann_R(M)}$ such that $(0:_Mr)\subseteq Im f^n+Ker f^n$ for some $n$, whence if $f$ is monic, then $f$ is nil-epic.
\item [(b)] If $M$ is a nil-Noetherian module, then there exists $r \in \sqrt{Ann_R(M)}$ such that $Im f^n\cap Ker f^n\subseteq rM$ for some $n$, whence if $f$ is epic, then $f$ is nil-monic.
\end{itemize}
\end{thm}
\begin{proof}
(a) Clearly, we have
$$
Im f \supseteq Im f^2\supseteq \cdots.
$$
Thus by assumption, there exists $r \in \sqrt{Ann_R(M)}$ such that $Im f^n \cap (0:_Mr) \subseteq Im f^{2n}$ for some positive integer $n$. Now, let $x \in  (0:_Mr)$. Then $rx=0$ and so $rf(x)=0$. Hence,  $f(x) \in  (0:_Mr)$. This implies that $f^n(x) \in  (0:_Mr)$. Therefore, $f^n(x) \in  Im f^n \cap (0:_Mr) \subseteq Im f^{2n}$. Thus $f^n(x)=f^{2n}(y)$. Clearly,
$$
x=f^n(y)+(x-f^n(y))\in Im f^n+Ker f^n.
$$
Finally, if $f$ is monic, then $Ker f^n=0$. This implies that $(0:_Mr)\subseteq Im f^n\subseteq Im f$.

(b) Since
$$
Ker f \subseteq Ker f^2\subseteq \cdots,
$$
by assumption, there exists $r \in \sqrt{Ann_R(M)}$ such that $Ker f^{2n}\subseteq Ker f^n+rM$ for some positive integer $n$. Now, let $x \in Im f^n\cap Ker f^n$. Then $x=f^n(y)$ and $f^n(x)=0$. Thus $f^{2n}(y)=0$. Therefore, $y \in Ker f^{2n}\subseteq  Ker f^n+rM$. Thus
$$
x=f^n(y) \in  f^n(Ker f^n)+f^n(rM)=rf^n(M)\subseteq rM.
$$
Finally, if $f$ is epic, then $Im f^n=M$. It follows that $Ker f\subseteq Ker f^n\subseteq rM$, as needed.
\end{proof}

\textbf{There is no conflict of interest}

\end{document}